\documentclass[12pt]{amsart}
\usepackage{amssymb}
\usepackage{amsmath}
\usepackage{mathrsfs}
\usepackage{amsthm}
\usepackage{enumerate}
\newtheorem{theorem}{Theorem}[section]
\newtheorem{lemma}[theorem]{Lemma}
\newtheorem{corollary}[theorem]{Corollary}
\numberwithin{equation}{section}
\begin{document}
\allowdisplaybreaks
\title{$C^{*}$-algebras isomorphically representable on $l^{p}$}
\author{March T.~Boedihardjo}
\address{Department of Mathematics, University of California, Los Angeles, CA 90095-1555}
\email{march@math.ucla.edu}
\begin{abstract}
Let $p\in(1,\infty)\backslash\{2\}$. We show that every homomorphism from a $C^{*}$-algebra $\mathcal{A}$ into $B(l^{p}(J))$ satisfies a compactness property where $J$ is any set. As a consequence, we show that a $C^{*}$-algebra $\mathcal{A}$ is isomorphic to a subalgebra of $B(l^{p}(J))$, for some set $J$, if and only if $\mathcal{A}$ is residually finite dimensional.
\end{abstract}
\maketitle
\section{Introduction}
For $1\leq p<\infty$ and a set $J$, let $l^{p}(J)$ be the space $\{f\colon J\to\mathbb{C}\colon\sum_{j\in J}|f(j)|^{p}<\infty\}$ with norm $\|f\|=(\sum_{j\in J}|f(j)|^{p})^{\frac{1}{p}}$. Two Banach algebras $\mathcal{A}_{1}$ and $\mathcal{A}_{2}$ are {\it isomorphic} if there exist a bijective homomorphism $\phi\colon\mathcal{A}_{1}\to\mathcal{A}_{2}$ and $C>0$ such that
\[\frac{1}{C}\|a\|\leq\|\phi(a)\|\leq C\|a\|,\]
for all $a\in\mathcal{A}_{1}$. The algebras $\mathcal{A}_{1}$ and $\mathcal{A}_{2}$ are {\it isometrically isomorphic} if moreover, $\phi$ can be chosen so that $\|\phi(a)\|=\|a\|$ for all $a\in\mathcal{A}_{1}$.

Gardella and Thiel \cite{Gardella} showed that for $p\in[1,\infty)\backslash\{2\}$, a $C^{*}$-algebra $\mathcal{A}$ is isometrically isomorphic to a subalgebra of $B(l^{p}(J))$, for some set $J$, if and only if $\mathcal{A}$ is commutative. So it is natural to consider the question whether this result holds if we relax the condition of isometrically isomorphic to isomorphic. In this paper, we show that for $p\in(1,\infty)\backslash\{2\}$, a $C^{*}$-algebra $\mathcal{A}$ is isomorphic to a subalgebra of $B(l^{p}(J))$, for some set $J$, if and only if $\mathcal{A}$ is residually finite dimensional (Corollary \ref{lprfd}). We prove this by showing that every homomorphism from a $C^{*}$-algebra $\mathcal{A}$ into $B(l^{p}(J))$ satisfies a compactness property (Theorem \ref{main}).

The proofs of the main results Theorem \ref{main} and Corollary \ref{lprfd} in this paper are quite different from the proof of Gardella-Thiel's result. Lamperti's characterization \cite{Lamperti} of isometries on $L^{p}$, for $p\neq 2$, plays a crucial role in the proof of Gardella-Thiel's result, while uniform convexity of $l^p$, for $1<p<\infty$, and an argument in probability that imitates the proof of Khintchine's inequality \cite[Theorem 2.b.3]{Lindenstrauss}, for $p=1$, are used in the proof of Theorem \ref{main}.
\section{Main results and proofs}
Throughout this paper, the scalar field is $\mathbb{C}$; for algebras $\mathcal{A}_{1}$ and $\mathcal{A}_{2}$, a {\it homomorphism} $\phi\colon\mathcal{A}_{1}\to\mathcal{A}_{2}$ is a bounded linear map such that $\phi(a_{1}a_{2})=\phi(a_{1})\phi(a_{2})$ for all $a_{1},a_{2}\in\mathcal{A}$; for an element $a$ of a $C^{*}$-algebra, $|a|=\sqrt{a^{*}a}$; the algebra of bounded linear operators on a Banach space $\mathcal{X}$ is denoted by $B(\mathcal{X})$ and the dual of $\mathcal{X}$ is denoted by $\mathcal{X}^{*}$; for $1\leq p\leq\infty$, the $l^{p}$ direct sum of Banach spaces $\mathcal{X}_{\alpha}$, for $\alpha\in\Lambda$, is denoted by $(\oplus_{\alpha\in\Lambda}\mathcal{X}_{\alpha})_{l^{p}}$. Two Banach spaces $\mathcal{X}_{1}$ and $\mathcal{X}_{2}$ are {\it isomorphic} if there is an invertible operator $S\colon\mathcal{X}_{1}\to\mathcal{X}_{2}$. A $C^{*}$-algebra $\mathcal{A}$ is {\it residually finite dimensional} if for every $a\in\mathcal{A}$, there is a $*$-representation $\phi$ of $\mathcal{A}$ on a finite dimensional space such that $\phi(a)\neq 0$.

\begin{theorem}\label{main}
Let $p\in(1,\infty)\backslash\{2\}$. Let $J$ be a set. Let $\mathcal{A}$ be a $C^{*}$-algebra. Let $\phi\colon\mathcal{A}\to B(l^{p}(J))$ be a homomorphism. Then
\begin{enumerate}[(i)]
\item the norm closure of $\{\phi(a)x\colon a\in\mathcal{A},\,\|a\|\leq 1\}$ in $l^{p}(J)$ is norm compact for every $x\in l^{p}(J)$; and
\item $\mathcal{A}/\text{ker }\phi$ is a residually finite dimensional $C^{*}$-algebra.
\end{enumerate}
\end{theorem}
\begin{corollary}\label{lprfd}
Let $p\in(1,\infty)\backslash\{2\}$. A $C^{*}$-algebra $\mathcal{A}$ is isomorphic to a subalgebra of $B(l^{p}(J))$, for some set $J$, if and only if $\mathcal{A}$ is residually finite dimensional.
\end{corollary}
Theorem \ref{main} and Corollary \ref{lprfd} will be proved at the end of this section after a series of lemmas are proved. Theorem \ref{main} has an easier proof when $\phi$ is contractive. Indeed, if $\phi\colon\mathcal{A}\to B(l^{p}(J))$ is a contractive homomorphism, then the range of $\phi$ is in the algebra of diagonal operators on $l^{p}(J)$ by \cite[Proposition 2.12]{Phillips} (or by \cite[Lemma 5.2]{Gardella} when $J$ is countable). Thus, $\{\phi(a)x\colon a\in\mathcal{A},\,\|a\|\leq 1\}$ is norm relatively compact, for every $x\in l^{p}(J)$, and $\mathcal{A}/\text{ker }\phi$ is commutative.

It is not known if Theorem \ref{main} and Corollary \ref{lprfd} hold for $p=1$. However, throughout their proofs, we use, in an essential way, the assumption that $p$ is in the reflexive range. For example, in the proof of Theorem \ref{main}(i), we use the fact that every bounded sequence in $l^{p}(J)$ has a weakly convergent subsequence. In the proof of Corollary \ref{lprfd}, we use a classical result of Pe{\l}czy\'nski that the $l^{p}$ direct sum of finite dimensional Hilbert spaces is isomorphic to $l^{p}(J)$ for some set $J$. This result of Pe{\l}czy\'nski holds only when $p$ is in the reflexive range.

The structure of the proof of Theorem \ref{main}(i) goes as follows: If the closure of $\{\phi(a)x_{0}\colon a\in\mathcal{A},\,\|a\|\leq 1\}$ is not compact for some $x_{0}\in l^{p}(J)$, then we can find a bounded sequence in $(b_{k})_{k\in\mathbb{N}}$ in $\mathcal{A}$ such that $\phi(b_{k})x_{0}\to 0$ weakly, as $k\to\infty$, and $\inf_{k\in\mathbb{N}}\|\phi(b_{k})x_{0}\|>0$. Assume that $p>2$. In Lemma \ref{premain1}, we show that $\phi(b_{k})\to 0$ weakly implies that $\omega(b_{k}^{*}b_{k})\to 0$ for all positive linear functional $\omega:\mathcal{A}\to\mathbb{C}$ of the form $\omega(a)=y_{0}^{*}(\phi(a)x_{0})$. This is proved by considering $\sum_{k=1}^{n}\delta_{k}b_{k}$ for random $\delta_{1},\ldots,\delta_{n}$ in $\{-1,1\}$ and by exploiting $p>2$. Lemma \ref{omega} says that when $y_{0}^{*}\in(l^{p}(J))^{*}$ is suitably chosen, $\omega(b_{k}^{*}b_{k})\to 0$ implies that $\|\phi(b_{k})x_{0}\|\to 0$, which contradicts with $\inf_{k\in\mathbb{N}}\|\phi(b_{k})x\|>0$. This is proved by using uniform convexity of $l^{p}(J)$.

Theorem \ref{main}(ii) follows from Theorem \ref{main}(i) by using a GNS type construction and a classical result about compact unitary representations of groups on Hilbert spaces.

The following two lemmas are needed for the proof of Lemma \ref{premain1}.
\begin{lemma}\label{pdreverse}
Let $\mathcal{A}$ be a unital $C^{*}$-algebra. Let $a\in\mathcal{A}$. Then there exists a sequence $(c_{n})_{n\in\mathbb{N}}$ in $\mathcal{A}$ such that $\|c_{n}\|\leq 1$ for all $n\in\mathbb{N}$ and $|a|=\lim_{n\to\infty}c_{n}a$.
\end{lemma}
\begin{proof}
Without loss of generality, we may assume that $\|a\|\leq 1$. For $n\in\mathbb{N}$, define $g_{n}\in C[0,1]$ by
\[g_{n}(x)=\begin{cases}\frac{1}{\sqrt{x}},&\frac{1}{n}\leq x\leq 1\\n\sqrt{n}x,&0\leq x\leq\frac{1}{n}\end{cases}.\]
Take $c_{n}=g_{n}(a^{*}a)a^{*}$. Then $c_{n}c_{n}^{*}=g_{n}(a^{*}a)a^{*}ag_{n}(a^{*}a)$. Note that
\[xg_{n}(x)^{2}=\begin{cases}1,&\frac{1}{n}\leq x\leq 1\\n^{3}x^{3},&0\leq x\leq\frac{1}{n}\end{cases}.\]
Thus, $0\leq xg_{n}(x)^{2}\leq 1$ for all $x\in[0,1]$ and so $0\leq c_{n}c_{n}^{*}\leq 1$. Hence $\|c_{n}\|\leq 1$.

We have
\[xg_{n}(x)=\begin{cases}\sqrt{x},&\frac{1}{n}\leq x\leq 1\\n\sqrt{n}x^{2},&0\leq x\leq\frac{1}{n}\end{cases}\]
and so $|xg_{n}(x)-\sqrt{x}|\leq\frac{1}{\sqrt{n}}$ for all $x\in[0,1]$. Since $c_{n}a=g_{n}(a^{*}a)a^{*}a$, it follows that $\|c_{n}a-\sqrt{a^{*}a}\|\leq\frac{1}{\sqrt{n}}$. Thus, the result follows.
\end{proof}
\begin{lemma}\label{holder}
Let $\mathcal{A}$ be a unital $C^{*}$-algebra. Let $\omega$ be a positive linear functional on $\mathcal{A}$. Let $a\in\mathcal{A}$. If $a\geq 0$ then
\[\omega(a^{2})\leq\omega(a)^{\frac{2}{3}}\omega(a^{4})^{\frac{1}{3}}.\]
\end{lemma}
\begin{proof}
There exists a measure $\mu$ on $[0,\|a\|]$ such that
\[\omega(f(a))=\int f(x)\,d\mu(x),\]
for all $f\in C[0,\|a\|]$. So
\[\omega(a^{2})=\int x^{2}\,d\mu(x)\leq\left(\int x\,d\mu(x)\right)^{\frac{2}{3}}\left(\int x^{4}\,d\mu(x)\right)^{\frac{1}{3}}=\omega(a)^{\frac{2}{3}}\omega(a^{4})^{\frac{1}{3}}.\]
\end{proof}
\begin{lemma}\label{premain1}
Let $2<p<\infty$. Let $J$ be a set. Let $\mathcal{A}$ be a unital $C^{*}$-algebra. Let $\phi\colon\mathcal{A}\to B(l^{p}(J))$ be a unital homomorphism. Let $x_{0}\in l^{p}(J)$. Let $y_{0}^{*}$ be a bounded linear functional on $l^{p}(J)$. Define $\omega\colon\mathcal{A}\to\mathbb{C}$ by
\[\omega(a)=y_{0}^{*}(\phi(a)x_{0}),\]
for $a\in\mathcal{A}$. Assume that $\omega$ is a positive linear functional. Let $(b_{k})_{k\in\mathbb{N}}$ be a sequence in $\mathcal{A}$ such that $\|b_{k}\|\leq 1$ for all $k\in\mathbb{N}$ and $\phi(b_{k})x_{0}\to 0$ weakly as $k\to\infty$. Then $\omega(b_{k}^{*}b_{k})\to 0$ as $k\to\infty$.
\end{lemma}
\begin{proof}
By contradiction, suppose that $\omega(b_{k}^{*}b_{k})$ does not converge to 0. Passing to a subsequence, we have that there exists $\gamma>0$ such that $\omega(b_{k}^{*}b_{k})\geq\gamma$ for all $k\in\mathbb{N}$.

Since $\|\phi(b_{k})x_{0}\|\leq\|\phi\|\|x_{0}\|$ and $\phi(b_{k})x_{0}\to 0$ weakly, passing to a further subsequence, we may assume that there are $z_{1},z_{2},\ldots$ in $l^{p}(J)$ with disjoint supports such that $\|z_{k}\|\leq\|\phi\|\|x_{0}\|$ and $\|\phi(b_{k})x_{0}-z_{k}\|\leq\frac{1}{2^{k}}$ for all $k\in\mathbb{N}$.

Let $n\in\mathbb{N}$. For each $\delta=(\delta_{1},\ldots,\delta_{n})\in\{-1,1\}^{n}$, let
\[a_{\delta}=\left|\sum_{k=1}^{n}\delta_{k}b_{k}\right|\in\mathcal{A}.\]
By Lemma \ref{holder},
\[\omega(a_{\delta}^{2})\leq\omega(a_{\delta})^{\frac{2}{3}}\omega(a_{\delta}^{4})^{\frac{1}{3}}.\]
Thus,
\[\mathbb{E}\omega(a_{\delta}^{2})\leq[\mathbb{E}\omega(a_{\delta})]^{\frac{2}{3}}[\mathbb{E}\omega(a_{\delta}^{4})]^{\frac{1}{3}},\]
where $\mathbb{E}$ denotes expectation over $\delta=(\delta_{1},\ldots,\delta_{n})$ uniformly distributed on $\{-1,1\}^{n}$.

Note that
\begin{eqnarray*}
\mathbb{E}\omega(a_{\delta}^{2})&=&\mathbb{E}\omega\left(\left(\sum_{k=1}^{n}\delta_{k}b_{k}\right)^{*}
\left(\sum_{k=1}^{n}\delta_{k}b_{k}\right)\right)\\&=&
\mathbb{E}\omega\left(\sum_{1\leq j,k\leq n}\delta_{j}\delta_{k}b_{j}^{*}b_{k}\right)=\sum_{1\leq j,k\leq n}\mathbb{E}(\delta_{j}\delta_{k})\omega(b_{j}^{*}b_{k})=\sum_{k=1}^{n}\omega(b_{k}^{*}b_{k})\geq n\gamma.
\end{eqnarray*}
Therefore,
\begin{equation}\label{intermb}
n\gamma\leq[\mathbb{E}\omega(a_{\delta})]^{\frac{2}{3}}[\mathbb{E}\omega(a_{\delta}^{4})]^{\frac{1}{3}}.
\end{equation}
We have
\[a_{\delta}^{4}=\left[\left(\sum_{k=1}^{n}\delta_{k}b_{k}\right)^{*}\left(\sum_{k=1}^{n}\delta_{k}b_{k}\right)\right]^{2}=
\sum_{1\leq k_{1},\ldots,k_{4}\leq n}\delta_{k_{1}}\delta_{k_{2}}\delta_{k_{3}}\delta_{k_{4}}b_{k_{1}}^{*}b_{k_{2}}b_{k_{3}}^{*}b_{k_{4}}.\]
Since $\|b_{k}\|\leq 1$, it follows that
\[\mathbb{E}\omega(a_{\delta}^{4})=\sum_{1\leq k_{1},\ldots,k_{4}\leq n}\mathbb{E}(\delta_{k_{1}}\delta_{k_{2}}\delta_{k_{3}}\delta_{k_{4}})\omega(b_{k_{1}}^{*}b_{k_{2}}b_{k_{3}}^{*}b_{k_{4}})\leq\sum_{1\leq k_{1},\ldots,k_{4}\leq n}\mathbb{E}(\delta_{k_{1}}\delta_{k_{2}}\delta_{k_{3}}\delta_{k_{4}}).\]
Note that $\mathbb{E}(\delta_{k_{1}}\delta_{k_{2}}\delta_{k_{3}}\delta_{k_{4}})=0$ unless the following occurs:
\[(k_{1}=k_{2}\text{ and }k_{3}=k_{4})\text{ or }(k_{1}=k_{3}\text{ and }k_{2}=k_{4})\text{ or }(k_{1}=k_{4}\text{ and }k_{2}=k_{3}).\]
Thus, $\mathbb{E}\omega(a_{\delta}^{4})\leq 3n^{2}$. So by (\ref{intermb}), we have $n\gamma\leq 3^{\frac{1}{3}}n^{\frac{2}{3}}[\mathbb{E}\omega(a_{\delta})]^{\frac{2}{3}}$. Hence,
\begin{equation}\label{intermb2}
\mathbb{E}\omega(a_{\delta})\geq\frac{\gamma^{\frac{3}{2}}}{3^{\frac{1}{2}}}n^{\frac{1}{2}}.
\end{equation}
Fix $\delta\in\{-1,1\}^{n}$. By Lemma \ref{pdreverse},
\[\omega(a_{\delta})=\omega\left(\left|\sum_{k=1}^{n}\delta_{k}b_{k}\right|\right)\leq
\sup_{c\in\mathcal{A},\,\|c\|\leq 1}\left|\omega\left(c\sum_{k=1}^{n}\delta_{k}b_{k}\right)\right|.\]
For $c\in\mathcal{A}$ with $\|c\|\leq 1$,
\begin{eqnarray*}
\left|\omega\left(c\sum_{k=1}^{n}\delta_{k}b_{k}\right)\right|&=&
\left|y_{0}^{*}\left(\phi(c)\left(\sum_{k=1}^{n}\delta_{k}\phi(b_{k})x_{0}\right)\right)\right|\\&\leq&
\|y_{0}^{*}\|\|\phi\|\left\|\sum_{k=1}^{n}\delta_{k}\phi(b_{k})x_{0}\right\|\\&\leq&
\|y_{0}^{*}\|\|\phi\|\left(\left\|\sum_{k=1}^{n}\delta_{k}z_{k}\right\|+\sum_{k=1}^{n}\frac{1}{2^{k}}\right)\\&\leq&
\|y_{0}^{*}\|\|\phi\|(\|\phi\|\|x_{0}\|n^{\frac{1}{p}}+1),
\end{eqnarray*}
where the last two inequalities follow from the fact that $z_{1},z_{2},\ldots$ have disjoint supports, $\|z_{k}\|\leq\|\phi\|\|x_{0}\|$ and $\|\phi(b_{k})x_{0}-z_{k}\|\leq\frac{1}{2^{k}}$. Thus, $\omega(a_{\delta})\leq\|y_{0}^{*}\|\|\phi\|(\|\phi\|\|x_{0}\|n^{\frac{1}{p}}+1)$ for all $\delta\in\{-1,1\}^{n}$. So by (\ref{intermb2}),
\[\frac{\gamma^{\frac{3}{2}}}{3^{\frac{1}{2}}}n^{\frac{1}{2}}\leq\|y_{0}^{*}\|\|\phi\|(\|\phi\|\|x_{0}\|n^{\frac{1}{p}}+1).\]
Since $n$ can be chosen to be arbitrarily large and $p>2$, an absurdity follows.
\end{proof}
For $1<p<2$, we have the following result, where the order of $b_{k}^{*}$ and $b_{k}$ are switched, by using the dual $l^{p}$ space in Lemma \ref{premain1}.
\begin{lemma}\label{premain2}
Let $1<p<2$. Let $J$ be a set. Let $\mathcal{A}$ be a unital $C^{*}$-algebra. Let $\phi\colon\mathcal{A}\to B(l^{p}(J))$ be a unital homomorphism. Let $x_{0}\in l^{p}(J)$. Let $y_{0}^{*}$ be a bounded linear functional on $l^{p}$. Define $\omega\colon\mathcal{A}\to\mathbb{C}$ by
\[\omega(a)=y_{0}^{*}(\phi(a)x_{0}),\]
for $a\in\mathcal{A}$. Let $(b_{k})_{k\in\mathbb{N}}$ be a sequence in $\mathcal{A}$ such that $\|b_{k}\|\leq 1$ for all $k\in\mathbb{N}$ and that the sequence $y_{0}^{*}\circ\phi(b_{k})$ of bounded linear functionals on $l^{p}(J)$ converges to $0$ weakly as $k\to\infty$. Assume that $\omega$ is a positive linear functional. Then $\omega(b_{k}b_{k}^{*})\to 0$ as $k\to\infty$.
\end{lemma}
\begin{proof}
Let $\mathcal{A}_{1}$ be the unital $C^{*}$-algebra consisting of the same elements as $\mathcal{A}$ but with reverse order multiplication
\[a\cdot b=ba.\]
Define a unital homomorphism $\phi_{1}\colon\mathcal{A}_{1}\to B((l^{p}(J))^{*})$ by
\[\phi_{1}(a)y^{*}=y^{*}\circ\phi(a),\]
for all $a\in\mathcal{A}_{1},\,y^{*}\in(l^{p}(J))^{*}$. Define $\omega_{1}\colon\mathcal{A}_{1}\to\mathbb{C}$ by
\[\omega_{1}(a)=\omega(a)=x_{0}^{**}(\phi(a)y_{0}^{*}),\]
for all $a\in\mathcal{A}_{1}$, where $x_{0}^{**}$ is the image of $x_{0}$ in the bidual $(l^{p})^{**}$. By Lemma \ref{premain1}, the result follows.
\end{proof}
The following two lemmas are needed for the proof of Lemma \ref{omega}.
\begin{lemma}[\cite{Clarkson}]\label{unifconv}
Let $1<p<\infty$. Let $J$ be a set. For every $\epsilon>0$, there exists $\gamma>0$ such that for all $x,y\in l^{p}(J)$ satisfying $\|x\|,\|y\|\leq 1$ and $\|x+y\|>2-\gamma$, we have $\|x-y\|<\epsilon$.
\end{lemma}
\begin{lemma}[\cite{Russo}]\label{russodye}
Let $\mathcal{A}$ be a unital $C^{*}$-algebra. Then the closed unital ball of $\mathcal{A}$ is the closed convex hull of the set of all unitary elements of $\mathcal{A}$.
\end{lemma}
\begin{lemma}\label{omega}
Let $1<p<\infty$. Let $J$ be a set. Let $\mathcal{A}$ be a unital $C^{*}$-algebra. Let $\phi\colon\mathcal{A}\to B(l^{p}(J))$ be a unital homomorphism. Let $x_{0}\in l^{p}(J)$. Then there exists $y_{0}^{*}\in(l^{p}(J))^{*}$ such that $\omega\colon\mathcal{A}\to\mathbb{C}$,
\[\omega(a)=y_{0}^{*}(\phi(a)x_{0}),\quad a\in\mathcal{A},\]
defines a positive linear functional and for every $\epsilon>0$, there exists $\gamma>0$ such that whenever $a\in\mathcal{A}$ satisfies $\|a\|\leq 1$ and $\omega(a^{*}a)<\gamma$, we have $\|\phi(a)x_{0}\|<\epsilon$.
\end{lemma}
\begin{proof}
Let $\mathcal{U}(\mathcal{A})$ be the set of all unitary elements of $\mathcal{A}$. Let $(v_{n})_{n\in\mathbb{N}}$ be a sequence in $\mathcal{U}(\mathcal{A})$ such that
\[\lim_{n\to\infty}\|\phi(v_{n})x_{0}\|=\sup_{u\in\mathcal{U}(\mathcal{A})}\|\phi(u)x_{0}\|.\]
For each $n\in\mathbb{N}$, let $x_{n}^{*}$ be a bounded linear functional on $l^{p}(J)$ such that $\|x_{n}^{*}\|=1$ and $x_{n}^{*}(\phi(v_{n})x_{0})=\|\phi(v_{n})x_{0}\|$. Then $x_{n}^{*}\circ\phi(v_{n})$ is a bounded sequence in $(l^{p}(J))^{*}$. Passing to a subsequence, we may assume that $x_{n}^{*}\circ\phi(v_{n})$ converges weakly to a bounded linear functional $y_{0}^{*}\in(l^{p}(J))^{*}$ as $n\to\infty$. Thus, $\omega\colon\mathcal{A}\to\mathbb{C}$,
\[\omega(a)=y_{0}^{*}(\phi(a)x_{0})=\lim_{n\to\infty}x_{n}^{*}(\phi(v_{n}a)x_{0}),\]
for $a\in\mathcal{A}$, defines a bounded linear functional on $\mathcal{A}$. Note that
\[\omega(1)=\lim_{n\to\infty}x_{n}^{*}(\phi(v_{n})x_{0})=\lim_{n\to\infty}\|\phi(v_{n})x_{0}\|=\sup_{u\in\mathcal{U}(\mathcal{A})}
\|\phi(u)x_{0}\|,\]
and for every $u_{0}\in\mathcal{U}(\mathcal{A})$,
\[|\omega(u_{0})|=\lim_{n\to\infty}|x_{n}^{*}(\phi(v_{n}u_{0})x_{0})|\leq\sup_{u\in\mathcal{U}(\mathcal{A})}
\|\phi(u)x_{0}\|.\]
So by Lemma \ref{russodye}, we have $\|\omega\|\leq\sup_{u\in\mathcal{U}(\mathcal{A})}\|\phi(u)x_{0}\|$. Thus, $\omega(1)=\|\omega\|$ and hence $\omega$ is a positive linear functional.

By contradiction, suppose that there are $\epsilon>0$ and a sequence $(a_{k})_{k\in\mathbb{N}}$ in $\mathcal{A}$ such that $\|a_{k}\|\leq 1$ and $\|\phi(a_{k})x_{0}\|\geq\epsilon$ for all $k\in\mathbb{N}$ and $\omega(a_{k}^{*}a_{k})\to 0$ as $k\to\infty$. We have
\[\|a_{k}\|\geq\frac{\|\phi(a_{k})x_{0}\|}{\|\phi\|\|x_{0}\|}\geq\frac{\epsilon}{\|\phi\|\|x_{0}\|},\]
for all $k\in\mathbb{N}$. For $k\in\mathbb{N}$, let $b_{k}=\frac{a_{k}}{\|a_{k}\|}$. We have $\|b_{k}\|=1$ and $\|\phi(b_{k})x_{0}\|\geq\epsilon$ for all $k\in\mathbb{N}$ and $\omega(b_{k}^{*}b_{k})\to 0$ as $k\to\infty$.

Since $\|x_{n}^{*}\|=1$,
\begin{align*}
&\liminf_{n\to\infty}\|\phi(v_{n})\phi(1-|b_{k}|)x_{0}+\phi(v_{n})x_{0}\|\\
\geq&\liminf_{n\to\infty}[x_{n}^{*}(\phi(v_{n})\phi(1-|b_{k}|)x_{0})+x_{n}^{*}(\phi(v_{n})x_{0})]\\
=&\omega(1-|b_{k}|)+\omega(1)=2\omega(1)-\omega(|b_{k}|).
\end{align*}
Thus,
\[\liminf_{n\to\infty}\|\phi(v_{n})\phi(1-|b_{k}|)x_{0}+\phi(v_{n})x_{0}\|\geq2\omega(1)-\omega(|b_{k}|).\]
But
\[\|\phi(v_{n})\phi(1-|b_{k}|)x_{0}\|\leq\sup_{b\in\mathcal{A},\,\|b\|\leq 1}\|\phi(b)x_{0}\|\|1-|b_{k}|\|\leq
\sup_{u\in\mathcal{U}(\mathcal{A})}\|\phi(u)x_{0}\|=\omega(1)\]
and $\|\phi(v_{n})x_{0}\|\leq\omega(1)$ for all $n\in\mathbb{N}$. Take
\[x=\frac{1}{\omega(1)}\phi(v_{n})\phi(|b_{k}|)x_{0}\text{ and }y=\frac{1}{\omega(1)}\phi(v_{n})x_{0}\]
in Lemma \ref{unifconv} and note that $\omega(|b_{k}|)\leq\omega(b_{k}^{*}b_{k})^{\frac{1}{2}}\omega(1)^{\frac{1}{2}}\to 0$ as $k\to\infty$. We have
\[\lim_{k\to\infty}\limsup_{n\to\infty}\left\|\phi(v_{n})\phi(1-|b_{k}|)x_{0}-\phi(v_{n})x_{0}\right\|=0.\]
Thus,
\[\lim_{k\to\infty}\limsup_{n\to\infty}\|\phi(v_{n})\phi(|b_{k}|)x_{0}\|=0.\]
So $\|\phi(|b_{k}|)x_{0}\|\to 0$ as $k\to\infty$. Since $b_{k}=b_{k}(|b_{k}|+\frac{1}{k})^{-1}(|b_{k}|+\frac{1}{k})$ and $\|b_{k}(|b_{k}|+\frac{1}{k})^{-1}\|\leq 1$, it follows that $\|\phi(b_{k})x_{0}\|\to 0$ as $k\to\infty$ which contradicts with $\|\phi(b_{k})x_{0}\|\geq\epsilon$.
\end{proof}
\begin{proof}[Proof of Theorem \ref{main}(i)]
Without loss generality, we may assume that $\mathcal{A}$ is unital by extending $\phi$ to a homomorphism from the unitization of $\mathcal{A}$ into $B(l^{p}(J))$. We may also assume that $\phi$ is unital since $\phi(1)$ is an idempotent on $l^{p}(J)$ and the range of every idempotent on $l^{p}(J)$ is isomorphic to $l^{p}(J_{0})$ for some set $J_{0}$ \cite{Pelczynski}, \cite{Johnson}.

Let $x_{0}\in l^{p}$. Let $(a_{k})_{k\in\mathbb{N}}$ be a sequence in $\mathcal{A}$ such that $\|a_{k}\|\leq\frac{1}{2}$ for all $k\in\mathbb{N}$. We need to show that $(\phi(a_{k})x_{0})_{k\in\mathbb{N}}$ has a norm convergent subsequence.

Case 1: $p>2$

Passing to a subsequence, we may assume that $(\phi(a_{k})x_{0})_{k\in\mathbb{N}}$ converges weakly to an element of $l^{p}(J)$. Thus, $\phi(a_{k_{1}}-a_{k_{2}})x_{0}\to 0$ weakly as $k_{1},k_{2}\to\infty$.

By Lemma \ref{premain1}, we have $\lim_{k_{1},k_{2}\to\infty}\omega((a_{k_{1}}-a_{k_{2}})^{*}(a_{k_{1}}-a_{k_{2}}))=0$ for every positive linear functional $\omega\colon\mathcal{A}\to\mathbb{C}$ of the form $\omega(a)=y_{0}^{*}(\phi(a)x_{0})$ for $a\in\mathcal{A}$. By Lemma \ref{omega}, we have $\lim_{k_{1},k_{2}\to\infty}\|\phi(a_{k_{1}}-a_{k_{2}})x_{0}\|=0$. So $(\phi(a_{k})x_{0})_{k\in\mathbb{N}}$ is norm convergent.

Case 2: $p<2$

Passing to a subsequence, we may assume that $(y_{0}^{*}\circ\phi(a_{k}^{*}))_{k\in\mathbb{N}}$ converges weakly to an element of $(l^{p}(J))^{*}$. Thus, $y^{*}\circ\phi(a_{k_{1}}^{*}-a_{k_{2}}^{*})\to 0$ weakly as $k_{1},k_{2}\to\infty$ for every $y^{*}\in(l^{p}(J))^{*}$.

By Lemma \ref{premain2}, we have $\lim_{k\to\infty}\omega((a_{k_{1}}^{*}-a_{k_{2}}^{*})(a_{k_{1}}^{*}-a_{k_{2}}^{*})^{*})=0$ for every positive linear functional $\omega\colon\mathcal{A}\to\mathbb{C}$ of the form $\omega(a)=y_{0}^{*}(\phi(a)x_{0})$ for $a\in\mathcal{A}$. By Lemma \ref{omega}, we have $\lim_{k_{1},k_{2}\to\infty}\|\phi(a_{k_{1}}-a_{k_{2}})x_{0}\|=0$. So $(\phi(a_{k})x_{0})_{k\in\mathbb{N}}$ is norm convergent.
\end{proof}
\begin{lemma}[\cite{Kerr}, Theorem 2.24]\label{cdecomp}
Let $G$ be a group. Let $\mathcal{H}$ be a Hilbert space. Let $\varphi\colon G\to B(\mathcal{H})$ be a unital homomorphism such that $\varphi(g)$ is unitary for all $g\in G$. If $\{\varphi(g)x\colon g\in G\}$ is norm precompact in $\mathcal{H}$ for all $x\in\mathcal{H}$, then $\mathcal{H}$ is the direct sum of some finite dimensional subspaces $\mathcal{H}_{\alpha}$, for $\alpha\in\Lambda$, such that $\mathcal{H}_{\alpha}$ is invariant under $\varphi(g)$ for all $\alpha\in\Lambda$ and $g\in G$.
\end{lemma}
\begin{proof}[Proof of Theorem \ref{main}(ii)]
As in the proof Theorem \ref{main}(i), we may assume that $\mathcal{A}$ is unital and $\phi$ is unital. We may also assume that $\text{ker }\phi=\{0\}$. Let $a_{0}\neq 0$. There exists $x_{0}\in l^{p}(J)$ such that $\phi(a_{0})x_{0}\neq 0$. By Lemma \ref{omega}, there exists $y_{0}^{*}\in(l^{p}(J))^{*}$ such that $\omega\colon\mathcal{A}\to\mathbb{C}$,
\[\omega(a)=y_{0}^{*}(\phi(a)x_{0}),\]
for $a\in\mathcal{A}$, defines a positive linear functional and $\omega(a_{0}^{*}a_{0})\neq 0$.

Equip $\mathcal{A}$ with the positive semidefinite sesquilinear form
\[\langle a,b\rangle=\omega(b^{*}a),\]
for $a,b\in\mathcal{A}$. Consider the ideal $\mathcal{A}_{0}=\{a\in\mathcal{A}\colon\langle a,a\rangle=0\}$ of $\mathcal{A}$. Let $\mathcal{H}$ be the completion of the quotient space $\mathcal{A}/\mathcal{A}_{0}$. Then $\mathcal{H}$ is a Hilbert space. For each $a\in\mathcal{A}$, we can define a bounded linear operator on $\mathcal{H}$ by sending $b+\mathcal{A}_{0}$ to $ab+\mathcal{A}_{0}$ for $b\in\mathcal{A}$. So $\eta\colon\mathcal{A}\to B(\mathcal{H})$,
\[\eta(a)(b+\mathcal{A}_{0})=ab+\mathcal{A}_{0},\]
for $a,b\in\mathcal{A}$, defines a unital $*$-homomorphism. We have
\begin{eqnarray*}
\|\eta(a_{1})(b+\mathcal{A}_{0})-\eta(a_{2})(b+\mathcal{A}_{0})\|&=&\omega(b^{*}(a_{1}-a_{2})^{*}(a_{1}-a_{2})b)\\&=&
y_{0}^{*}(\phi(b^{*}(a_{1}-a_{2})^{*}(a_{1}-a_{2})b)x_{0})\\&\leq&
\|y_{0}^{*}\|\|\phi\|\|b^{*}\|\|a_{1}-a_{2}\|\|\phi(a_{1}-a_{2})\phi(b)x_{0}\|,
\end{eqnarray*}
for all $a_{1},a_{2},b\in\mathcal{A}$. By Theorem \ref{main}(i), we have that $\{\phi(a)x_{0}\colon a\in\mathcal{A},\,\|a\|\leq 1\}$ is norm precompact so $\{\eta(a)(b+\mathcal{A}_{0})\colon a\in\mathcal{A},\,\|a\|\leq 1\}$ is norm precompact for all $b\in\mathcal{A}$. Let $\mathcal{U}(\mathcal{A})$ be the set of all unitary elements of $\mathcal{A}$. By Lemma \ref{cdecomp}, we have that $\mathcal{H}$ is the direct sum of some finite dimensional subspaces $\mathcal{H}_{\alpha}$, for $\alpha\in\Lambda$, such that $\mathcal{H}_{\alpha}$ is invariant under $\eta(u)$ for all $\alpha\in\Lambda$ and $u\in\mathcal{U}(\mathcal{A})$. Note that $\mathcal{H}_{\alpha}$ is thus invariant under $\eta(a)$ for all $a\in\mathcal{A}$.

Since $\omega(a_{0}^{*}a_{0})\neq 0$, we have $\eta(a_{0})\neq 0$. So $\eta(a_{0})\neq 0$ on $\mathcal{H}_{\alpha_{0}}$ for some $\alpha_{0}\in\Lambda$. Thus, $\mathcal{A}$ is residually finite dimensional.
\end{proof}
\begin{proof}[Proof of Corollary \ref{lprfd}]
One direction follows from Theorem \ref{main}. For the other direction, suppose that $\mathcal{A}$ is a residually finite dimensional $C^{*}$-algebra. Then there is a collection $(\phi_{\alpha})_{\alpha\in\Lambda}$ of $*$-representations of $\mathcal{A}$ on finite dimensional Hilbert spaces $\mathcal{H}_{\alpha}$ such that $\|a\|=\sup_{\alpha\in\Lambda}\|\phi_{\alpha}(a)\|$ for all $a\in\mathcal{A}$. Define $\phi\colon\mathcal{A}\to B((\oplus_{\alpha\in\Lambda}\mathcal{H}_{\alpha})_{l^{p}})$ by $\phi=\oplus_{\alpha\in\Lambda}\phi_{\alpha}$. Thus $\phi$ is a norm preserving homomorphism. However, it is a classical result of Pe{\l}czy\'nski \cite{Pelczynski} that for $1<p<\infty$, the $l^{p}$ direct sum of finite dimensional Hilbert spaces is isomorphic to $l^{p}(J)$ for some set $J$. Therefore, $\mathcal{A}$ is isomorphic to a subalgebra of $B(l^{p}(J))$, via the map $a\mapsto S\phi(a)S^{-1}$ where $S\colon(\oplus_{\alpha\in\Lambda}\mathcal{H}_{\alpha})_{l^{p}}\to l^{p}(J)$ is any invertible operator.
\end{proof}
{\bf Acknowledgements:} The author is grateful to the referee for some suggestions that improved the exposition. The author is supported by NSF DMS-1856221.

\end{document}